\documentclass{amsart}

\pdfoutput=1

\usepackage{amsmath, mathtools, comment, bbold}
\usepackage{amsfonts}
\usepackage{natbib}
\usepackage{amssymb}
\usepackage{amsthm}
\usepackage{tikz-cd}
\usetikzlibrary{decorations.pathmorphing}
\usepackage{tikz}
\usepackage{amsxtra}
\usepackage{wrapfig}
\usepackage{color}
\usepackage{url}
\usepackage[hidelinks]{hyperref}
\usepackage{mathrsfs}
\usepackage{stmaryrd}
\usepackage{setspace}
\usepackage[utf8]{inputenc}
\usepackage[T1]{fontenc}
\usepackage{textcomp}
\usepackage{microtype}
\usepackage{cleveref}

\newtheorem{thm}{Theorem}
\newtheorem{prop}[thm]{Proposition}

\newtheorem{lem}[thm]{Lemma}
\theoremstyle{definition}
\newtheorem{defn}[thm]{Definition}

\theoremstyle{remark}
\newtheorem{ex}[thm]{Example}
\newtheorem{rmk}[thm]{Remark}

\newcommand{\cat}[1]{{\mathbf{#1}}}

\newcommand{\into}{\hookrightarrow}

\newcommand{\per}{{\ensuremath{\cat{per}}}\kern 1pt}

\DeclareMathOperator{\id}{id}

\let\hom\relax\newcommand{\hom}{\mathrm{Hom}}

\DeclareMathOperator{\map}{Hom}
\DeclareMathOperator{\ind}{ind}

\newcommand{\V}{\mathcal{V}}

\newcommand{\C}{\mathcal{C}}
\newcommand{\D}{\mathcal{D}}

\newcommand{\Z}{\mathbb{Z}}

\newcommand{\R}{{\mathrm{\normalfont\mathbb{R}}}}

\newcommand{\rmap}{\R\kern -1.5pt \map}

\numberwithin{equation}{section}

\newcommand*{\1}{\text{\usefont{U}{bbold}{m}{n}1}}

\title{How to invert well-pointed endofunctors}
\author{Matt Booth}


\address{\noindent{Department of Mathematics,
Imperial College London,
SW7 2AZ, United Kingdom}\newline
Heilbronn Institute for Mathematical Research,
Bristol, BS8 1UG, United Kingdom}
	\email{matt.booth@imperial.ac.uk}

\thanks{
This work was supported by the Additional Funding Programme for Mathematical
Sciences, delivered by EPSRC (EP/V521917/1) and the Heilbronn Institute for Mathematical Research.\\
I would like to thank Sebastian Opper and Julie Symons for helpful discussions.}

\begin{document}
	\maketitle
    \begin{abstract}
        In this short note we observe that Kelly's transfinite construction of free algebras yields a way to invert well-pointed endofunctors. In enriched settings, this recovers constructions of Keller, Seidel, and Chen--Wang. We also relate this procedure to localisation by spectra and to Heller's stabilisation.
    \end{abstract}


\section{Enriched preliminaries}
Throughout we will let $(\mathcal{V},\otimes, \1)$ be a bicomplete closed symmetric monoidal category. We write the internal hom-objects as $\V(x,y)\in \V$ and the homsets as $\hom_\V(x,y)\in \mathbf{Set}$. We will assume that $\hom_\V(\1,-)$ is faithful, so that we can regard the objects of $\mathcal{V}$ as sets with extra structure (we call such monoidal categories \textbf{concrete}). We will moreover assume that $\1$ is compact, so that limits and filtered colimits in $\V$ are created in $\mathbf{Set}$.\footnote{If $\V$ is locally presentable, then $\hom_\V(\1,-)$ has a left adjoint \cite[1.66]{AdamekRosicky}. If $\otimes$ in addition preserves compact objects, then we are in the setup of \cite{kellyref}.} The reader who does not care for generalities can imagine $\V$ to be $\mathbf{Set}$, $\mathbf{Vect}$, or $\mathbf{dgVect}$. If $\C$ is a $\V$-category, we denote the enriched hom-objects by $\C(x,y)\in \V$ and the underlying homsets by $\hom_\C(x,y)\in \mathbf{Set}$. 

If $\D$ is an ordinary category, recall that it has an \textbf{ind-category} $\ind\D$ whose objects are given by diagrams $X:J\to \D$ where $J$ is small and filtered, and morphisms are given by $\hom_{\ind \D}(X,Y)\coloneqq \varprojlim_i\varinjlim_j\hom_\D(X_i,Y_j)$. Note that $\ind\D$ is an accessible category, and is locally finitely presentable provided that $\D$ is cocomplete (e.g.~ \cite[11.1]{isaksen}). There is an embedding $\D \into \ind\D$ sending an object $x$ to the diagram $\ast\xrightarrow{x}\D$. If $\D$ has filtered colimits, this has an adjoint given by $\varinjlim$.

If $\C$ is a $\V$-category, then since limits and filtered colimits in $\V$ are created in $\mathbf{Set}$ then the exact same formulas provide a canonical $\V$-enrichment for $\ind\C$. We denote this enriched ind-category by $\hat \C$, so that the underlying category of $\hat \C$ is $\ind\C$. Again, there is a $\V$-functor $\C \into \hat \C$, which is universal in the sense that any $\V$-functor $F:\C\to\D$ extends to a $\V$-functor $\hat F:\hat\C \to \hat\D$ by requiring it to commute with formal filtered colimits. By construction, $\hat F$ is accessible (by which we mean simply that the underlying functor is accessible).

There is a deep theory of enriched accessible categories and the closely related notion of enriched ind-completions \cite{kellyref, BNQ, LT1, LT2}. When the enriching category $\V$ has nontrivial homotopy theory, one also wants enriched ind-categories that take this homotopy theory into account: when $\V = \mathbf{dgVect}$ such a \textit{homotopy ind-dg-completion} is given in \cite{GLSVdB}. In this note we take a more na\"ive approach.

	\section{Well-pointed endofunctors}
    The arguments in this section are all essentially due to Kelly \cite{kelly}, although we circumvent some of the issues encountered there by passing to ind-categories. Our presentation here was heavily influenced by \cite{nlab}. From now on, $\V$ is a concrete bicomplete closed symmetric monoidal category with compact unit. All categories, functors, etc.~ will be enriched over $\V$. A \textbf{pointed endofunctor} on a category $\C$ is a natural transformation $\theta:\id \to \Omega$ of functors on $\C$. Say that $(\Omega,\theta)$ is \textbf{well-pointed} if $\theta\Omega = \Omega\theta$: for all $X$ we have $\theta_{\Omega X}=\Omega(\theta_X)$ as maps $\Omega X \to \Omega^2 X$. An $\Omega$\textbf{-algebra} is an object $X$ together with a map $\Omega X \to X$ such that the composition $X \xrightarrow{\theta_X} \Omega X \to X$ is the identity. There is an evident category of $\Omega$-algebras $\mathbf{Alg}(\Omega)$, constructed as a slice category.
    \begin{lem}If $\theta$ is well-pointed then an object $X$ admits the structure of an $\Omega$-algebra if and only if $\theta_X$ is invertible; in this case the algebra structure is unique.
			\end{lem}
			\begin{proof}
            This is \cite[Proposition 5.2]{kelly}. If $\theta_X$ is invertible then one takes the algebra structure map $\Omega X \to X$ to be its inverse. Conversely, if $f:\Omega X \to X$ is any morphism then well-pointedness yields a commutative diagram $$\begin{tikzcd}
					\Omega X \ar[r,"f"]\ar[d,"{\Omega \theta_X}"]& X \ar[d,"\theta_X "]\\ \Omega^2 X \ar[r,"\Omega f"]& \Omega X
					\end{tikzcd}$$
				which shows that $\theta_Xf = \Omega(f\theta_X)$. If $f$ is an algebra then this shows that $f$ is both a left and right inverse of $\theta_X$, and thus $\theta_X$ is invertible. It is clear that the algebra structure must be unique.
					\end{proof}
In particular, if $\theta$ is well-pointed then the category $\mathbf{Alg}(\Omega)$ is naturally a full subcategory of $\C$. By extending $(\Omega,\theta)$ to a well-pointed endofunctor $(\hat\Omega, \hat\theta)$ of $\hat\C$, we see that we may define a functor $\hat\Omega^\infty: \hat\C \to \hat\C$ by
$$\hat\Omega^\infty (X)\;\coloneqq\; \varinjlim\left(X \xrightarrow{\hat\theta_X} \hat\Omega X \xrightarrow{\hat\theta_{\Omega X}} \hat\Omega^2 X  \xrightarrow{\hat\theta_{\Omega^2 X}}\cdots\right)$$	where we take the colimit in the ind-category\footnote{If $X:J\to\C$ is a filtered diagram, then the colimit of the associated diagram $J\xrightarrow{X} \mathcal{C}\to\hat\C$ is precisely the ind-object $X$. One can easily prove this using the Yoneda lemma.}.
		
		\begin{thm}If $\C$ is cocomplete, then $\hat\Omega^\infty$ is a reflection of $\hat\C$ into $\mathbf{Alg}(\hat\theta)$.
			\end{thm}
			\begin{proof}
            This is \cite[Remark 6.3]{kelly}, which applies since $\hat \C$ is locally presentable and in particular well-copowered. The idea is simple: by construction $\hat\Omega$ is accessible, so for any ind-object $X$ we obtain a natural map $\hat\Omega \hat\Omega^\infty (X) \to \hat \Omega^\infty (X)$ that makes $\hat\Omega^\infty (X)$ into an $\hat\Omega$-algebra. It follows that $\hat\Omega^\infty$ is a reflection of $\hat\C$ into $\mathbf{Alg}(\hat \C)$.
				\end{proof}
From now on we assume that $\C$ is cocomplete. The following definition, at least in the enriched setting, is due to Wolff \cite{wolffMonads,wolff}:
         \begin{defn}
		A functor $F:\C \to \D$ \textbf{inverts} a natural transformation $\theta$ between endofunctors of $\C$ if for every $X$ in $\C$, the morphism $F(\theta_X)$ is an isomorphism. The \textbf{localisation of $\C$ along $\theta$} is the initial functor that inverts $\theta$; i.e.\ it is a functor $\gamma: \C \to \C'$ such that if $F:\C \to \D$ inverts $\theta$ then there exists a unique $F':\C' \to \D$ such that $F=F'\gamma$.
			\end{defn}	
Let $\Omega^\infty$ denote the composition $\C \into \hat\C \xrightarrow{\hat\Omega^\infty}\mathbf{Alg}(\hat\Omega)$. We have an isomorphism $\mathbf{Alg}(\hat\Omega)(\Omega^\infty X, \Omega^\infty Y) \cong \varprojlim_n \varinjlim_m {\C}(\Omega^n X, \Omega^m Y)$, since $\mathbf{Alg}(\hat\Omega)$ is a full subcategory of $\hat\C$. On the other hand we also have isomorphisms
 $$\mathbf{Alg}(\hat\Omega)(\Omega^\infty X, \Omega^\infty Y) \;\cong\; \hat\C(X, \Omega^\infty Y)\;\cong\; \varinjlim_m \C(X, \Omega^m Y)$$
 which will be of more use to us. Write $L_\Omega (\C) \into \hat\C$ for the essential image of $\Omega^\infty$. 
\begin{thm}
	$\Omega^\infty:\C\to L_\Omega (\C)$ is the localisation of $\C$ at $\theta$.
	\end{thm}
\begin{proof}
	Suppose $F:\C \to \D$ is a functor such that every $F(\theta_X)$ is an isomorphism. Extend $F$ to a functor $\hat F:\hat\C \to \hat \D$ and consider the composition $\hat F \Omega^\infty$. By construction we have $$\hat F\Omega^\infty X \;\cong\; \varinjlim \left(FX \to F\Omega X \to F \Omega^2 X \to\cdots\right)$$but by assumption, every map in this colimit is an isomorphism, and so we see that $\hat F \Omega^\infty \cong F$. In other words, $F$ factors through the essential image of $\Omega^\infty$. We need to check that the factoring map $\hat F$ is unique. So let $G:L_\Omega (\C) \to \D$ be any functor such that $G\Omega^\infty = F$. Pick $X\in L_\Omega (\C)$. Since $L_\Omega (\C)$ is defined to be the essential image of $\Omega^\infty$, there must be $X'\in \C$ such that $X\cong \Omega^\infty X'$, and hence $G(X)=F(X') = \hat F (X)$. Let $$G_{\Omega^\infty X, \Omega^\infty Y}:\;L_\Omega (\C)(\Omega^\infty X, \Omega^\infty Y) \longrightarrow \D(G\Omega^\infty X, G\Omega^\infty Y)$$be the component maps of $G$. Replacing $L_\Omega (\C)(\Omega^\infty X, \Omega^\infty Y)$ by a colimit as above, we see that $G_{\Omega^\infty X, \Omega^\infty Y}$ is an inverse limit of maps of the form $$\phi_m:\;\C(X, \Omega^m Y) \longrightarrow \D(G\Omega^\infty X, G\Omega^\infty Y).$$We have a commutative diagram in $\V$ (cf.~ the proof of \cite[Lemma 1.1]{seidel})
	$$
	\begin{tikzcd}
		\C(X, \Omega^m Y) \ar[rr,"\phi_m"]\ar[dr, "G\Omega^\infty_{X,\Omega^mY}"]&& \D(G\Omega^\infty X, G\Omega^\infty Y)\ar[dl, "\psi"]\\
		& \D(G\Omega^\infty X, G\Omega^\infty\Omega^m Y)
	\end{tikzcd}
	$$where $\psi$ is induced by the canonical morphism $Y \to \Omega^m Y$. Because $\Omega^\infty$ inverts $\theta$, it follows that $\psi$ is an isomorphism. In particular, $G_{\Omega^\infty X, \Omega^\infty Y}$ is the inverse limit of the system of maps $G\Omega^\infty_{X,\Omega^m Y} = \hat F\Omega^\infty_{X,\Omega^m Y}$. Running the same argument for $\hat F$ shows that $G$ must be naturally isomorphic to $\hat F$.
	\end{proof}

	\section{Examples}
Here we let $k$ be a field; all categories will be linear over $k$.

\begin{ex}
    Let $\C$ be a $k$-linear category and $T:F\to\mathrm{id}$ a well-copointed\footnote{i.e.\ $T^\mathrm{op}$ is well-pointed; Seidel uses the term \textbf{ambidextrous} \cite{seidel}.} endofunctor on $\C$. Running our constructions in $\C^\mathrm{op}$ yields a localisation $L_F(\C^\mathrm{op})$ that agrees with Seidel's construction \cite{seidel}. In particular, if $\C$ is a pretriangulated dg category and $F$ is a dg functor, then $L_F(\C^\mathrm{op})$ can be identified as the dg quotient of $\C^\mathrm{op}$ by the pretriangulated subcategory spanned by those objects that are annihilated by some power of $F$ \cite[Lemma 1.3]{seidel}.
\end{ex}

\begin{ex}\label{scatex}
    Let $\C$ be a dg-$k$-category and $\theta:\id\to \Omega$ a well-pointed dg endofunctor on $\C$. Then $L_\Omega(\C)$ is precisely the localisation $\mathcal{S}\C$ constructed by Chen and Wang \cite[\S6]{CWKW}\footnote{The motivating example of \cite{CWKW} is the case when $\C$ is the \textbf{Yoneda dg category} of an algebra $A$ and $\Omega$ is the \textbf{noncommutative differential forms} functor; the localisation $\mathcal{S}\C$ is then a model for the dg singularity category of $A$.}. Hence if $\C $ is pretriangulated then $L_\Omega(\C)$ is a model for the dg quotient $\C/\mathbf{thick}\left(\mathrm{cone(\theta_X): X\in \C}\right)$ by \cite[Theorem 6.4]{CWKW}. Note that $L_\Omega(\C)$ is a strictification of Keller's ind-categorical description of the dg quotient \cite{KelQuot}. Indeed, if $\D$ is a pretriangulated dg subcategory of $\C$ then the dg quotient $\C/\D$ can be described as the subcategory of $\hat \C$ on those ind-objects $X$ right orthogonal to $\D$ and which fit into an exact triangle $c \to X \to Y \to$ with $c\in \C$ and $Y \in \hat\D$, as made clear in  \cite[4.9]{Drinfeld}. This provides a high-level viewpoint on some computations of stable Ext made by the author in \cite[Theorem 6.4.6]{MR4214397}.
\end{ex}

\begin{ex}
    Let $\mathcal{A}$ be a dg-$k$-category and $F:\mathcal{A} \to \mathcal{A}$ a dg endofunctor. Define a new dg category $\mathcal{A}_F$ with the same objects as $\mathcal{A}$, and hom-complexes given by $\mathcal{A}_F(X,Y)\coloneqq \oplus_n\mathcal{A}(F^nX,Y)$. The composition of $F^iX \to Y$ and $F^jY\to Z$ is given by $F^{i+j}X \to F^iY \to Z$. The resulting endofunctor $F$ of $\mathcal{A}_F$ is well-pointed, by the natural transformation with components $\mathrm{id}_{FX}\in\mathcal{A}_F(X,FX)$; this is in fact the universal way to make $F$ well-pointed. Then $L_F(\mathcal{A}_F)$ is Keller's dg orbit category \cite[5.1]{keller}. Note that $L_F(\mathcal{A}_F)$ need not be pretriangulated, even if $\mathcal{A}$ was.
\end{ex}

\section{Spectra}
As above, all categories, functors, etc.\, remain enriched over $\mathcal{V}$. Let $\C$ be a category and $\Omega$ an endofunctor\footnote{When $\V=\mathbf{Ab}$ then this is precisely the notion of \textbf{looped category} from \cite{beligiannis}.}. A \textbf{spectrum} is a sequence $X_0, X_1,X_2,\ldots$ of objects in $\C$ with morphisms $\sigma_n:X_{n} \to \Omega X_{n+1}$. A spectrum is an \textbf{$\Omega$-spectrum} when the morphisms $\sigma_n$ are all isomorphisms\footnote{One sometimes calls the first kind of object a \textbf{prespectrum} and the other simply a \textbf{spectrum}.}. There is an evident category $\mathrm{Sp}_\Omega(\C)$ of spectra together with a full subcategory $\smash{\underline{\mathrm{Sp}}}_{\Omega}(\C)$ of $\Omega$-spectra. Since limits in $\mathcal{V}\text{-}\mathbf{Cat}$ are computed pointwise, there is an equivalence of categories $$\smash{\underline{\mathrm{Sp}}}_{\Omega}(\C) \cong \varprojlim \left(\cdots \xrightarrow{\Omega}\C\xrightarrow{\Omega}\C\xrightarrow{\Omega}\C\right)$$and when $\mathcal{V} = \mathbf{Set}$ then ${\mathrm{Sp}}_\Omega(\C)$ can also be obtained as the analogous 2-limit taken in $\mathbf{Cat}$.\footnote{Presumably a similar statement holds for general $\mathcal{V}$, possibly with some additional assumptions.} Observe that the map $X\mapsto X_n$ which assigns a spectrum its $n^\text{th}$ level can be regarded as a functor ${\mathrm{Sp}}_\Omega(\C) \to \C$.

There is a shift endofunctor $S$ of $\mathrm{Sp}_\Omega(\C)$ given on sequences by $(SX)_i=X_{i+1}$. The $\Omega$ functor extends to an endofunctor of ${\mathrm{Sp}}_\Omega(\C)$, and one can easily check that $\Omega S = S\Omega$. There is a natural transformation $\sigma:\id \to \Omega S$ defined on sequences by $\sigma_{n}: X_n \to \Omega X_{n+1} = \Omega S(X_n)$, making $\Omega S$ into a well-pointed endofunctor.\footnote{The argument showing that $\Omega S$ is well-pointed is precisely the argument which shows that $\sigma$ is a well-defined natural transformation.}

Let $\mathcal{L}$ denote the localisation $L_{\Omega S}({\mathrm{Sp}}_{\Omega }\C)$, which is a a subcategory of the category of ind-spectra $\widehat{\mathrm{Sp}}_\Omega(\C)$. This category comes equipped with a localisation functor $\Omega^\infty S^\infty\coloneqq (\Omega S)^\infty:{\mathrm{Sp}}_{\Omega }(\C) \to{\mathcal{L}}$. Since $\Omega $ and $S$ commute, so do $\hat\Omega$ and $\hat S$, and hence they are mutually inverse functors on $\mathcal{L}$.

Observe that there is a natural fully faithful functor $\iota:\widehat{\mathrm{Sp}}_\Omega(\C) \to \mathrm{Sp}_{\hat\Omega}(\hat\C)$ defined as follows. If $X:J \to {\mathrm{Sp}}_\Omega(\C)$ is an ind-spectrum, then $(\iota X)_n$ is the ind-object $J \xrightarrow{X} {\mathrm{Sp}}_\Omega(\C) \xrightarrow{(-)_n}\C$. The connecting maps are obtained analogously.\footnote{More abstractly, a spectrum is a certain kind of pro-object, and the natural comparison functor $\mathrm{indpro}\C \to \mathrm{proind}\C$ gives the map from ind-spectra to spectra in ind-objects.}

 We refer to the composition $\iota \Omega^\infty S^\infty$ as the \textbf{spectrification} functor; by construction its image lies in the subcategory $\smash{\underline{\mathrm{Sp}}}_{\hat\Omega}(\hat\C)$.
One can easily compute that if $X$ is a spectrum, we have
$(\iota \Omega^\infty S^\infty X)_n \cong \varinjlim\left(X_n \to \Omega X_{n+1} \to \Omega^2X_{n+2}\to\cdots\right)$, where we take the filtered colimit in $\hat\C$. The structure maps are induced from those of $X$.

\begin{ex}
   When $\C$ has filtered colimits, the composition $$\varinjlim \ \circ \  \iota \Omega^\infty S^\infty : \mathrm{Sp}_\Omega(\C) \to \smash{\underline{\mathrm{Sp}}}_{\Omega}(\C)$$ is (an enriched version of) the classical spectrification appearing in e.g.\ \cite{LMMS}.
\end{ex}

\begin{ex}\label{Theta}
    Suppose that the endofunctor $\Omega$ was actually well-pointed, by a natural transformation $\theta$. This yields a functor $\Theta:\C \to {\mathrm{Sp}}_{\Omega }(\C)$ defined by $\Theta(X)_n = X$, with the structure maps $\sigma_n:X \to \Omega X$ given by $\theta$. Then the spectrification of $\Theta (X)$ has at all levels the localisation $\Omega^\infty (X)$. 
\end{ex}

	\begin{ex}
	    Suppose that the endofunctor $\Omega$ admits a left adjoint $\Sigma$. This yields a functor $\Sigma^\infty:\C \to {\mathrm{Sp}}_{\Omega }(\C)$ by putting $\Sigma^\infty(X)_n= \Sigma^nX$. The structure map $\Sigma^nX \to \Omega\Sigma^{n+1}X$ is the adjunct of the identity map on $\Sigma^{n+1}$. Note that by composition with $\Omega^{n}$ this yields maps $\Omega^n\Sigma^nX \to \Omega^{n+1}\Sigma^{n+1}X$. Put $$\Omega^\infty \Sigma^\infty X\coloneqq \varinjlim\left(X \to \Omega \Sigma X \to \Omega^2\Sigma^2 X\to\cdots\right)$$ where again we take the filtered colimit in $\hat \C$. This construction is topologically known as the \textbf{free infinite loop space on} $X$. One can check that the $n^\text{th}$ level of the spectrification of $\Sigma^\infty X$ is precisely
        $\Omega^\infty \Sigma^{\infty}(\Sigma^nX)$, which recovers the classical topological fact that $\Omega^\infty \Sigma^\infty X$ is the zeroth level of the spectrification of $\Sigma^\infty X$.
	\end{ex}

\begin{rmk}
For the purposes of algebraic topology, especially constructing a symmetric monoidal smash product of spectra, the above approach is known to be completely inadequate \cite{lewis}. One either needs to use model categories of highly structured spectra, as in e.g.\, \cite{MMSS}, or use $\infty$-categories from the beginning, as in \cite{HA}. We note that similar constructions to that of this section in a homotopy-invariant setting have already been given in \cite[\S8]{Heller}.
\end{rmk}

\section{Stabilisation, cospectra, and comparisons}

Once again we work in the enriched setting. Let $\C$ be a category and $\Omega$ an endofunctor of $\C$. Following Heller \cite[\S1]{HellerStab}, we define a new category $\mathcal{S}_\Omega\mathcal{C}$, the \textbf{stabilisation} of $\C$, as follows. The objects are the pairs $(c,i)$ with $c\in \C$ and $i\in \Z$. The morphisms are defined to be $$\mathcal{S}_\Omega\mathcal{C}((c,i),(d,j))\coloneqq  \varinjlim_{k}\C(\Omega^{k+i}c,\Omega^{k+j}d)$$with composition inherited from $\C$. For brevity we will write $[-,-]$ for the hom-objects in $\mathcal{S}_\Omega\mathcal{C}$; with this notation we clearly have $[(c,i),(d,j)]\simeq [(c,i+l),(d,j+l)]$ for all $l\in \Z$. The functor $\Omega$ extends to the stabilisation by putting $\Omega(c,i) \coloneqq (\Omega c , i)$, and one can easily verify via the Yoneda lemma that there is a natural isomorphism $\Omega(c,i) \cong (c , i+1)$. In particular, $\Omega$ is an autoequivalence of $\mathcal{S}_\Omega\mathcal{C}$, with inverse $(c,i)\mapsto (c,i-1)$. There is an obvious functor $\C \to \mathcal{S}_\Omega\mathcal{C}$ sending $c$ to $(c,0)$, which is universal with respect to stabilising $\Omega$ \cite[Proposition 1.1]{HellerStab}.

Observe that there is a natural comparison map $\Phi:\smash{\underline{\mathrm{Sp}}}_{\Omega}(\C) \to \mathcal{S}_\Omega \C$ defined by sending a spectrum $X$ to the pair $(X_0,0)\cong (X_i,i)$.

\begin{prop}
    Suppose that $\theta:\id\to \Omega$ is a well-pointed endofunctor on a locally finitely presentable\footnote{One can remove this assumption by replacing $\C$ by $\hat \C$; for readability we refrain from doing this.} category $\C$. We denote by $\Omega^\infty:\C \to \C$ the corresponding localisation functor, with image $L_\Omega\C \into \C$.
    \begin{enumerate}
        \item The localisation $L_\Omega\C$ is a coreflective subcategory of $\mathcal{S}_\Omega\C$, with coreflection given by the functor $\eta$ which sends $(d,i)$ to $(\Omega^\infty (d),0) \cong (\Omega^\infty(d),n)$.
        \item The localisation $L_\Omega\C$ is a coreflective subcategory of $\smash{\underline{\mathrm{Sp}}}_{\Omega}(\C)$, with coreflection given by the functor $\varepsilon$ which sends a spectrum $X$ to the constant spectrum on $\Omega^\infty (X_0)$ (with structure maps as in \Cref{Theta}).
        \item There is a natural comparison map $\Psi: \mathcal{S}_\Omega \C \to \smash{\underline{\mathrm{Sp}}}_{\Omega}(\C)$ which sends $(c,i)$ to the constant spectrum on $\Omega^\infty (c)$ .
        \item There are natural isomorphisms $\Phi\Psi\cong \eta$ and $\Psi\Phi\cong \varepsilon$.
        \item The following are equivalent:
        \begin{itemize}
            \item $\Phi$ is an equivalence, with inverse $\Psi$.
            \item Both $\smash{\underline{\mathrm{Sp}}}_{\Omega}(\C)$ and $\mathcal{S}_\Omega \C$ are naturally equivalent to $L_\Omega\C$.
        \end{itemize}
    \end{enumerate}
\end{prop}
\begin{proof}
    For (1), the inclusion functor is the composition $L_\Omega\C \into \C \to \mathcal{S}_\Omega\C$; this is fully faithful since $\Omega^k\Omega^\infty \cong \Omega^\infty $ as functors on $\C$. For the coreflection, we compute
    $$[(\Omega^\infty c,0),(d,i)] \cong \varinjlim_k \C(\Omega^\infty c, \Omega^{k+i} d)\cong \varinjlim_k \C(\Omega^\infty c, \Omega^{k} d) \cong L_\Omega\C(\Omega^\infty c, \Omega^\infty d)$$where in the last step we use the natural isomorphism $\Omega^\infty\Omega^\infty\cong \Omega^\infty$. The proof of (2) is similar; here the inclusion functor is the composition $L_\Omega\C \into \C \xrightarrow{\Theta} \smash{\underline{\mathrm{Sp}}}_{\Omega}(\C)$ where $\Theta$ is the functor of \Cref{Theta}. For (3), since $\smash{\underline{\mathrm{Sp}}}_{\Omega}(\C)$ stabilises $\Omega$, the universal property of the stabilisation ensures the existence of $\Psi$ and the proof of \cite[Proposition 1.1]{HellerStab} yields the desired description. Claim (4) is a simple computation and claim (5) follows easily.
\end{proof}

\begin{rmk}
    When $\V=\mathbf{Set}$\footnote{More generally, this holds when $\V$ is a presheaf category (e.g.\ $\mathbf{sSet}$), since colimits in $\V$ are computed pointwise. In general, colimits in $\mathcal{V}\text{-}\mathbf{Cat}$ are nontrivial to compute \cite{wolff}.}, one can regard $\mathcal{S}_\Omega \C$ as the colimit of the diagram $\C\xrightarrow{\Omega}\C\xrightarrow{\Omega}\C\xrightarrow{\Omega}\cdots$, which one could call the category of \textbf{$\Omega$-cospectra}\footnote{To obtain the cospectra of \cite{lima}, one should instead take the corresponding 2-colimit. Presumably one can then adapt the arguments of the previous section to construct a cospectrification functor which replaces a cospectrum by an $\Omega$-cospectrum. Note that \cite{AnnalaIwasa} refers to the higher-categorical version of cospectra as \textbf{telescopes}.}. If $J$ denotes the doubly-infinite diagram $\cdots \xrightarrow{\Omega} \C \xrightarrow{\Omega} \C \xrightarrow{\Omega} \C \xrightarrow{\Omega} \cdots$ then we obtain a natural comparison map $ \smash{\underline{\mathrm{Sp}}}_{\Omega}(\C)\cong\varprojlim J \longrightarrow \varinjlim J \cong \mathcal{S}_\Omega \C$ which agrees with the comparison map $\Phi$ defined above. Hence, in this setting, $\Phi$ is an equivalence precisely when $\Omega$ \textbf{has eventual image duality} in the sense of \cite{leinster}. For more on the duality between spectra and cospectra, see \cite[\S4]{Grandis}. 
\end{rmk}

\begin{rmk}Suppose that $\theta:\mathrm{id} \to \Omega$ is a well-pointed endofunctor on $\C$. Although both $\smash{\underline{\mathrm{Sp}}}_{\Omega}(\C)$ and $\mathcal{S}_\Omega\C$ satisfy a universal property with respect to stabilising $\Omega$, neither construction need actually invert the map $\theta$.
\end{rmk}

\begin{rmk}
For certain left triangulated categories $(\C,\Omega)$, the stabilisation $\mathcal{S}_\Omega\C$ can be realised as a generalised singularity category \cite[Theorem 3.8]{beligiannis}, cf.\ \cite{Buchweitz, KellerVossieck, ChenWangNew}. Dually, for certain right triangulated categories, the costabilisation $\smash{\underline{\mathrm{Sp}}}_{\Omega}(\C)$ has a similar interpretation \cite[Theorem 3.11]{beligiannis}, cf.\ \cite{Grandis}.
\end{rmk}

\begin{footnotesize}
	\bibliographystyle{alpha}
	\bibliography{references.bib}
\end{footnotesize}

	\end{document}